\documentclass[a4paper]{amsart} 
\usepackage{amsmath}

\usepackage{amssymb}
\usepackage{verbatim}
\usepackage{amsthm}
\usepackage{mathrsfs}
\usepackage{cancel}
\usepackage{graphicx}
\usepackage{mathtools}
\usepackage[colorlinks,pdfdisplaydoctitle,linkcolor=blue,citecolor=red]{hyperref}
\usepackage{caption}
\usepackage{subcaption}
\usepackage{url}
\usepackage{epstopdf}
\usepackage{pgf,tikz}
\usepackage{bbm}

\renewcommand{\Delta}{\triangle}

\definecolor{darkblue}{rgb}{0,0,0.7}

\definecolor{darkgreen}{rgb}{0.01,0.75,0.24}

\def \Ee[#1]{\mathcal{E}^{\text{{#1}}}}
\def\R{\mathbf{R}}  

\def\bq{\mathbf{q}}

\def\pa[#1,#2]{\frac{\partial {#1}}{\partial {#2}} }

\def\idom[#1,#2,#3]{\int_{#1}\hspace{1pt} {#2} \hspace{1pt} \text{d}{#3}}
\def\res[#1,#2]{\left.{#1}\right|_{#2}}

\def\gt{\rightarrow}

\def\var[#1,#2]{\langle \delta \mathcal{E}^{\text{{#1}}}({#2}),v\rangle}
\def\vars[#1,#2,#3]{\langle \delta^2\mathcal{E}^{\text{{#1}}}({#2})v,{#3}\rangle}
\def\vard[#1,#2,#3,#4]{\langle \delta\mathcal{E}^{\text{{#1}}}({#2})-\delta\mathcal{E}^{\text{{#3}}}({#4}),v\rangle}

\def\P{\mathcal{P}}

\def\E{\mathbb{E}}

\def\mO{\mathcal{O}}
\def\mD{\mathcal{D}}
\def\mH{\mathcal{H}}
\def\N{\mathbb{N}}

\newcommand{\bB}{\mathbf{B}}

\newcommand{\A}{\mathcal{A}}
\newcommand{\mL}{\mathcal{L}}

\newcommand{\X}{\mathcal{X}}

\newcommand{\eps}{\varepsilon}

\newcommand{\bx}{\mathbf{x}}
\newcommand{\by}{\mathbf{y}}

\newcommand{\bp}{\mathbf{p}}

\newcommand{\be}{\begin{equation}}
\newcommand{\en}{\end{equation}}
\newcommand{\ben}{\begin{equation*}}
\newcommand{\enn}{\end{equation*}}
\newcommand{\bea}{\begin{aligned}}
\newcommand{\ena}{\end{aligned}}

\def\ba#1\ena{\begin{align}#1\end{align}}
\def\ban#1\enan{\begin{align*}#1\end{align*}}

\theoremstyle{plain}
\newtheorem{thm}{Theorem}[section]

\newtheorem{lem}[thm]{Lemma}

\newtheorem{prop}[thm]{Proposition}
\newtheorem{assumption}[thm]{Assumptions}

\theoremstyle{remark}
\newtheorem{rem}[thm]{Remark}

\numberwithin{equation}{section}

\begin{document}
\title[Geometric ergodicity of Langevin dynamics with Coulomb interactions]{Geometric ergodicity of Langevin dynamics with Coulomb interactions}

 \author[Y. Lu]{Yulong Lu}
 \address[Y. Lu]{Department of Mathematics, Duke University, Durham NC 27708, USA}
\email{yulonglu@math.duke.edu}

 \author[J. C. Mattingly]{Jonathan C. Mattingly}
 \address[J. C. Mattingly]{Departments of Mathematics and Statistical Science, Duke University, Durham NC 27708, USA}
\email{jonm@math.duke.edu}

\begin{abstract}
This paper is concerned with the long time behavior of Langevin dynamics of {\em Coulomb gases} in $\mathbf{R}^d$ with $d\geq 2$, that is
a second order system of Brownian particles driven by an external force and
a pairwise repulsive  Coulomb force. We prove that the system converges exponentially to
the unique Boltzmann-Gibbs invariant measure under a weighted total variation distance. The proof relies on
a novel construction of Lyapunov function for the Coulomb system.
\end{abstract}



\maketitle

\section{Introduction }
 \subsection{Model description}
In this paper, we study a classical Langevin dynamics for a system of $N$ particles in $\R^d, d \geq 2$
subject to a confining external force and interacting
through the Coulomb force. The evolution of particles is described
by the following stochastic differential equations (SDEs):
\be\label{eq:LGN}
\begin{aligned}
		\frac{d\bq}{dt} & = \nabla_\bp H(\bq,\bp)\\
		\frac{d\bp}{dt} & = -\gamma \nabla_\bp
                H(\bq,\bp)\, dt - \nabla_\bq H(\bq,\bp)\,
                dt + \sqrt{\frac{2\gamma  }{\beta}}\frac{d\bB}{dt},
\end{aligned}
\en
where $\bp = (p_1, \cdots, p_N)\in (\R^{d})^N$ and $ \bq = (q_1, \cdots, q_N)\in (\R^{d})^N$
are vectors of positions and momentum of the $N-$particles in the Euclidean space
$\R^d$ with $d \geq 2$ and  $\bB(t)$ is a vector of $N$ independent $d$-dimensional Brownian motions.
Here the constant $\gamma >0$ is a friction parameter, $\beta$
is an inverse temperature parameter.

The Hamiltonian (or energy) $H(\bq, \bp)$ is defined by
\be
\begin{aligned}\label{eq:H}
 H (\bq,\bp):=  \frac{1}{2} \sum_{i=1}^N |p_i|^2
 + \sum_{i=1}^N V(q_i) +  \frac{1}{2N}\sum_{1 \leq i\neq j\leq N}K(q_i -q_j),
\end{aligned}
\en
which is the sum of the kinetic energy $\frac{1}{2}\sum_{i=1}^N |p_i|^2$ and the potential energy  $U(\bq)$ defined by
\be\label{eq:U}
U(\bq) :=  \sum_{i=1}^N V(q_i) +  \frac{1}{2N}\sum_{1 \leq i\neq j\leq N}K(q_i -q_j).
\en
The potential energy consists of two parts: an external energy and an interacting energy.
Observe that the $N$-dependent prefactors appearing in the interaction term
in the definition of the Hamiltonian $H$  keep them balanced in the
same order $\mathcal{O}(N)$.
The function $V:\R^d \gt \R$ is an external confining potential, growing fast enough at infinity, and satisfying some extra conditions to be specified later. The pairwise interacting kernel $K:\R^d \gt \R$ is the Coulomb kernel in $\R^d$:
\be\label{eq:Coulomb}
K(x) = \begin{cases}
-\log |x| & \text{ if  } d= 2,\\
 |x|^{2-d} & \text{ if  }  d \geq 3.
\end{cases}
\en
 The Coulomb kernel $K$ is the fundamental solution of the Poisson equation in the free
physical space $\R^d$, i.e.
\begin{align*}
  - \Delta K = c_d \delta_0, \text{ with } c_d = \begin{cases}
                                             2\pi & \text{ if } d = 2\\
                                             (d-2)|\mathbb{S}^{d-1}| = \frac{d(d-2)\pi^{d/2}}{\Gamma(1+d/2)} & \text{ if } d \geq 3.
                                            \end{cases}
\end{align*}

It is easy to check that the Boltzmann-Gibbs measure $
\pi(d\bq \,d \bp)$ is an  invariant measure of the particle system \eqref{eq:LGN}:
\begin{align*}
  \pi(d\bq \,d \bp) =  P(d \bp)\otimes Q(d \bq)  = \frac{1}{Z}e^{-\beta H(\bq, \bp)}d\bq\, d \bp,
\end{align*}
where
\begin{align*}
  P (d\bp) = \frac{1}{Z_p} \exp\Big(-\frac{\beta}{2} |\bp|^2\Big) \text{ and }
Q (d\bq) = \frac{1}{Z_q} \exp\Big(-\beta U(\bq)\Big).
\end{align*}
The primary aim of this paper is  to prove that the particle system
\eqref{eq:LGN} converges exponentially fast to the (unique) invariant
measure $\pi$ as $t\gt \infty$. This result will be a consequence of a
classical Harris ergodic theorem which assumes the existence for a
minorizing measure and a Lyapunov function in the spirit of
\cite{meyn1993stability,meyn2012markov,hairer2011yet}. The minorizing measure is
the source of the probabilistic mixing while the Lyapunov function
ensures the dynamics does not escape to infinite. The system is
hypoellitpic which quickly implies the minorization condition \cite{mattingly2002ergodicity}.
 The main contribution of this paper is to construct a
Lyapunov function in the Coulomb setting. The construction used in
\cite{mattingly2002ergodicity}, does not seem to work for the Coulomb system. The approach used here is a
further refinement of the ideas developed in \cite{herzog2017ergodicity}.
This is further discussed in Section~\ref{sec:DiscussionWrgodicity}.

\subsection{Motivating and connected problems}

The dynamics \eqref{eq:LGN} is usually referred to as the inertial Langevin equation in the physical-chemistry literature.
It models for example the time evolution of a system of confining charged particles (e.g. electrons)
 interacting  each other through the repulsive Coulomb force in a thermal environment. When $\gamma\gt \infty$, the Langevin dynamics \eqref{eq:LGN}, after proper accelerating in time,  reduces to
 the following first order equation
\be\label{eq:OLGN}
\frac{d \bq}{dt} = -\nabla U(\bq) + \sqrt{2\beta^{-1}} \frac{d\bB}{dt}
\en
or more concretely
\be\label{eq:OLGN2}
\frac{d q_i}{dt} = -\nabla V(q_i) - \frac{1}{N}\sum_{j\neq i}^N \nabla K (q_i -q_j) +\sqrt{2\beta^{-1}}\frac{dB_i}{dt}, i = 1,2,\cdots, N.
\en
The dynamics \eqref{eq:OLGN} is usually called the overdamped limit of the Langevin equation \eqref{eq:LGN}.
A rigorous derivation of the overdamped limit can be found in \cite{stoltz2010free}.
Note that the equilibrium measure of the overdamped Langevin equation \eqref{eq:OLGN} is given by
\be\label{eq:QN}
Q(d\bq)  \propto \exp\Big(-\beta \Big(\sum_{i=1}^N V(q_i) + \frac{1}{2N} \sum_{1\leq i \neq j \leq N} K(q_i - q_j)\Big)\Big)d\bq.
\en
When $\beta = \beta_N = N\tilde{\beta}$, the equilibrium measure $Q(d\bq)$ is usually refered to as the {\em Coulomb gases} model.
To clarify the terminology, we will refer \eqref{eq:LGN} and \eqref{eq:OLGN} as the {\em dynamical} Coulomb gases, and the
 measure \eqref{eq:QN} as the {\em steady} Coulomb gases.
In recent years, there has been a growing interest in the
probabilistic properties of steady Coulomb gases in meso/microscopic regimes; see e.g. \cite{SS15,SS152,RS16,S16,LS18}
and the references therein. However, the long time behavior of dynamical Coulomb gases was less understood.
For the overdamped model \eqref{eq:OLGN} with a logarithmic $K$ in one dimension, the exponential convergence to the equilibrium can be established by using the
Bakry-\'Emery argument thanks to the convexity of the negative logarithmic function in 1D; see e.g. \cite{EY17}. The recent paper \cite{bolley2018dynamics}
studied the convergence to the equilibrium in two dimensions via the Poincar\'e inequality.

The Coulomb gas model \eqref{eq:QN} has many important connections with other models in mathematical physics. In 1D or 2D, the Coulomb gas model $Q$ is closely related to some famous random matrix models.  In fact, if the potential $V$ is quadratic, e.g. $V(q) = \frac{|q|^2}{2\beta}$, then under the scaling  $\beta = N\tilde{\beta}$, the measure $Q$ corresponds to the famous $\beta$-Hermite and $\beta$-Ginibre ensemble in 1D and 2D respectively. In particular, when $\tilde{\beta}=2$, the gas describes the eigenvalue distribution of a random $N\times N$ matrix of i.i.d. real and complex normal entries, whose limiting distribution are  the Wigner's semicircle law and the circular law  respectively. Because of the link between  the equilibrium measure $Q$ and the eigenvalue distribution of random matrix ensembles,  the overdamped dynamics  \eqref{eq:OLGN2}
and the under-damped dynamics \eqref{eq:OLGN2} can be used for Monte Carlo
simulation of random matrices.
The recent paper \cite{chafai2018simulating}
proposed an efficient Hybrid Monte Carlo (HMC) method for sampling Coulomb gases,
based on a Metropolis-Hastings algorithm with a proposal produced
by discretizing the second order dynamics \eqref{eq:LGN}. However, the convergences of the sampling algorithm proposed there
and the associated continuous dynamics \eqref{eq:LGN} have not been analyzed yet.
This paper provides a theoretical underpinning for the HMC algorithm of
\cite{chafai2018simulating} by showing that the \eqref{eq:LGN} converges to the equilibrium Gibbs measure $\pi$
exponentially fast. Finally, we would like to point out a link between the dynamical Coulomb gas and partial differential equations (PDEs). At the dynamical level, one expects that as $N\gt \infty$ the empirical measure of the SDEs \eqref{eq:LGN} converges to its mean field limit which is the so-called Vlasov-Possion-Fokker-Planck equation. In the case where the kernel is Lipschitz, this picture has been well-understood. However, rigourous justifications of the mean field limits of dynamics of interacting particles  with singular kernels, especially the Coulomb kernel is highly non-trivial and far from understood. We will not discuss this problem in the present paper, but we refer the interested readers to \cite{fournier2014propagation, jabin2018quantitative, serfaty2018mean,liu2016propagation,jabin2016mean,CCS18,huang2018mean} for some recent progress.

 \subsection{Discussion on ergodicity for the Langevin equation}
\label{sec:DiscussionWrgodicity}
The Langevin equation is an SDE with degenerate noise. In the last few decades,
both probabilistic and analytical methods have been developed for studying the long time behavior of the
Langevin dynamics and the corresponding kinetic Fokker Planck equation. At the PDE level, an important analytical tool
is based on the so-called {\em hypocoercivity} initiated by Villani \cite{dric2009hypocoercivity} and then later
advanced by Dolbeault, Mouhot and Schmeiser \cite{dolbeault2009hypocoercivity,dolbeault2015hypocoercivity}.
The framework of hypercoercivity was also  employed recently by \cite{conrad2010construction,grothaus2015hypocoercivity}
to prove ergodicity for Langevin particle systems with some mild singular potentials excluding the Coulomb potential.

At the probabilistic level, the geometric ergodicity of the Langevin dynamics with regular drifts has been
proved by \cite{talay2002stochastic,mattingly2002ergodicity,wu2001large} based on the
techniques of Meyn and Tweedie \cite{meyn1993stability,meyn2012markov} and Khasminskii \cite{khasminskii2011stochastic}.
These techniques have been formalized now a standard route of establishing ergodicity of Markov processes: the validity of
a local minorization condition  and the existence of  a Lyapunov function; see \cite{hairer2011yet}.
The local minorization condition leads to contractivity of any two coupled stochastic trajectories in some average sense if they both start from a local compact set $\mathcal{C}$.
The existence of a Lyapunov function guarantees that the process return  to the compact set $\mathcal{C}$ infinitely many often and  usually provides quantitative estimates on the excursion time outside the set $\mathcal{C}$.
In the case where the underlying diffusion is degenerate but hypoelliptic, such as the Langevin dynamics considered here,
the minorization condition can be justified by a standard control-type argument \cite{mattingly2002ergodicity}. However, constructing Lyapunov functions for
degenerate diffusions is more technical and is problem-dependent.

For Langevin dynamics, a natural candidate of Lyapunov function would be the Hamiltonian.
Yet, due to the absence of diffusion in the position-coordinate, the Hamiltonian fails to be a real Lyapunov function since
it gives no dissipation in the position-coordinate. In the case where the drift of Langevin dynamics is regular, adding a $cp\cdot q$
term into the Hamiltonian with a suitable small $c>0$
enables transferring
 the dissipation from  the momentum  variable to the position and hence
 leads to a Lyapunov function, see \cite{talay2002stochastic,mattingly2002ergodicity}.
However, this $p q$ trick does not work when the drift of Langevin dynamics is singular. Recently, Herzog and Mattingly
\cite{herzog2017ergodicity} (see also \cite{cooke2017geometric})
constructed a Lyapunov function for a Langevin particle system with a class of admissible singular interaction kernels, including
Lennard-Jones type kernels.
The key idea underlying their construction is to understand the high energy limit of the dynamics via apppriate space-time scaling.
Lyapunov functions can then be built by identifying the effective/dominant generator in the high energy limit. However,
the Coulomb kernel (at least in 1D/2D) can not be treated in the framework of
\cite{herzog2017ergodicity,cooke2017geometric} since the admissible condition of \cite{herzog2017ergodicity} (see also below) is violated.


In this work, we revisit the idea
of \cite{herzog2017ergodicity} and tailor an
explicit Lyapunov function for the dynamical Coulomb gas \eqref{eq:LGN}
in $\R^d$ with any $d\geq 2$. In fact, the same Lyapunov function also works for dynamical Log gases \cite{F10} in one dimensions,
but we restrict our attension here on $d\geq 2$ only for ease of exposition.

\subsection{Plan of the paper} The rest of the paper is organized as follows. In the next section, we set up the model
with some assumptions on the potential function and state
the main results of the paper. In Section \ref{sec:model1}, we develop
some heuristics by building a Lyapunov function for a simple two-body Coulomb system.
In Section \ref{sec:model2}, we extend this heuristic idea to general many body Coulomb systems.
The proofs of main results are given in Section \ref{sec:proof}. Finally, we conclude in Section \ref{sec:ext} with several extensions and outlook.

\section{Set-up and Main results}
Let us recall the SDE \eqref{eq:LGN} and rewrite it as
\be\label{eq:LG}
\begin{aligned}
		\frac{d\bq}{dt} & = \bp\\
		\frac{d\bp}{dt} & = -\gamma \bp\, dt - \nabla U(\bq)\, dt + \sqrt{2\gamma \beta^{-1}}\frac{d\bB}{dt},
\end{aligned}
\en
where
$$
U(\bq)  =  \sum_{i=1}^N V(q_i) + \frac{1}{2N}\sum_{1\leq i\neq j \leq N}  K (q_i -q_j) .
$$
Recall also the Hamiltonian $H(\bq,\bp)$ defined by \eqref{eq:H}:
$$
H(\bq, \bp) = \frac{1}{2} |\bp|^2 + \sum_{i=1}^N V(q_i) + \frac{1}{2N}\sum_{1\leq i\neq j \leq N}  K (q_i -q_j).
$$
We aim to show that the dynamics \eqref{eq:LG} converges to the Boltzmann-Gibbs measure
$$
\pi(d\bq d\bp) = \frac{1}{Z} e^{-\beta H(\bq, \bp)}d\bq d\bp.
$$

To this end, we need to make some technical assumptions on the external potential $V$.
\begin{assumption}\label{ass:v}
The function $V:\R^d\gt \R $ is at least twice differentiable and satisfies


(i) $V(q) \geq 0$ for all $q\in \R^d$ and there exist constants $c_1,c_2,c_3, M > 0$ such that
\be\label{eq:assv1}
 V(q) \geq c_1|q|^2 - M \ \forall q\in \R^d,
\en
and that
\be\label{eq:assv2}
 c_2 V(q) -M \leq \nabla V(q) \cdot q \leq c_3 V(q) + M \ \forall q\in \R^d.
\en

(ii) For any $\eps > 0$, there exists $M_\eps > 0$ such that

\be\label{eq:assv3}
|\nabla V(q)| \leq \eps V(q) + M_\eps \ \forall q\in \R^d.
\en
\end{assumption}

\begin{rem}
 The quadratic growth condition \eqref{eq:assv1} implies that $V$ is strongly confining and it is necessary
 for the $pq$ trick to work in constructing Lyapunov functions for Langevin dynamics, see e.g. \cite{mattingly2002ergodicity}.
 It also implies that
 $$
 Z = \iint_{\R^d \times \R^d}  e^{-\beta H(\bq,\bp)}d\bq\, d\bp < \infty.
 $$
  The inequalities in \eqref{eq:assv2} are standard dissipative conditions
 on $V$.
 The condition \eqref{eq:assv3} controlling  $\nabla V$ by $V$ itself will be useful for
 our explicit construction
 of Lyapunov function for the Coulomb dynamics.  Assumption \ref{ass:v} is by no means sharp, but proves to
be sufficient for the validity of our main results. In particular, the standard double well function
$V(q) = \frac{1}{4} (1 - |q|^2)^2$ fullfils Assumption \ref{ass:v}.

\end{rem}

We denote by  $\mD$ the domain of $U$ in $(\R^{d})^N$, i.e.
$$
\mD  := \{\bq \in(\R^{d})^N: U(\bq) < \infty \}.
$$
We also define the state space  $\X := \mD \times (\R^d)^N$.
Due to the singularity of $K$ and the growth condition on $V$ at  infinity, it is clear that
\be\label{eq:O}
\mD = \{\bq = (q_1, q_2,\cdots, q_N) \in(\R^{d})^N: q_i\neq q_j \text{ if } i\neq j \text{ and } q_i < \infty\ \forall i=1,2,\cdots, N \}.
\en
The boundary of $\mD$ is given by
$$
\partial \mD := \{\infty\}\cup \bigcup_{1\leq i\neq j\leq N} \{\bq = (q_1, q_2,\cdots, q_N) \in(\R^{d})^N: q_i = q_j\}.
$$
Similarly, we have the boundary of $\X$ given by
$$
\partial \X := \{\partial \mD \times (\R^d)^N\} \cup \{\mD \times \{\infty\} \}.
$$
In addition, we have the following lemma whose proof is trivial and hence omitted.
\begin{lem}

(i) The set $\mD$ is open and path connected;

(ii) The set $\mD_R := \{\bq \in(\R^{d})^N: U(\bq ) < R\}$ is precompact in $(\R^{d})^N$ for every $R > 0$.
\end{lem}

Due to the singularity of the Coulomb kernel, the existence and uniqueness  of a  global pathwise solution to
the SDE \eqref{eq:LG} does not automatically follow from a
standard solution theory of SDE based on local Lipschitz continuity.
In particular, it is not guaranteed a priori that two particles could not collide.
The next proposition shows the global existence and uniqueness of the process $\bx = (\bq, \bp)\in \X$ satisfying \eqref{eq:LG}
and thus demonstrates  that collisions between particles never occur in finite time.

\begin{prop}\label{prop:exist}
Given any initial condition $\bx (0) = (\bq(0), \bp(0))\in \X $, there exists a unique pathwise solution $\bx(t) = (\bq(t), \bp(t))$ satisfying  the equation \eqref{eq:LG}. Moreover, $\bx(t)\in \X$ for any finite time $t > 0$ almost surely.
\end{prop}

Proposition \ref{prop:exist} can be proved by just using the Hamiltonian $H$ as the Lyapunov function, but we will give a proof in Section \ref{sec:proof} based on
a new Lyapunov function tailored for the dynamics \eqref{eq:LG}; see Proposition \ref{prop:pw}.

Thanks to Proposition \ref{prop:exist}, one can define the Markov semigroup $P_t$ associated to the SDE \eqref{eq:LG} by
$$
(P_t \varphi) (\bq, \bp) := \E_{(\bq, \bp)}[ \varphi(\bq(t), \bp(t))]
$$
where $(\bq(t), \bp(t))$ is the solution process of \eqref{eq:LG} starting from  $(\bq, \bp)\in \X$ and $\varphi:\X \gt \R$ is a bounded measurable function. Let $\mathscr{P}_\X$ be the space of probability measures on $\X$.  We also define the dual operation of $P_t$ on probability measures by setting for a given probability measure $\nu\in \mathscr{P}_X$,
$$
(\nu P_t) (A) = \int_\X (P_t \mathbf{1}_A) (\bq, \bp) \nu (d\bq\, d\bp),
$$
where $A$ is a Boreal measurable subset of $\X$ and $\mathbf{1}_A$ is the indicator function of $A$.

The generator $\mL$ associated to the semigroup $P_t$ is given for every smooth test function $f: \X \gt \R$ by
\be\label{eq:generator}
\mL f = \bp \cdot \nabla_\bq f - \gamma \bp \cdot \nabla_\bp f - \nabla_\bq U(\bq) \cdot \nabla_\bp f + \gamma \beta^{-1}\Delta_\bp f.
\en
With the definition of the generator $\mL$, we introduce an important concept of {\em Lyapunov} function.
We say a function $W: \X \gt (0,\infty)$ is a {\em Lyapunov} function of the equation \eqref{eq:LG}
if $W(\bq, \bp) \gt \infty$ as $H(\bq, \bp)  \gt \infty$ and
there exist constants $\lambda, C > 0$ such that
\be\label{eq:Lyapunov}
\mL W \leq -\lambda W + C.
\en

The main result of the paper addresses the convergence to equilibrium for the dynamical system of Coulomb gas  \eqref{eq:LG}. Motivated by \cite{hairer2011yet, herzog2017ergodicity}, we define a weighted total variation distance of probability measures.  To be more specific, for any measurable function $W: \X \gt (0,\infty)$, we define the set $\mathscr{P}_W$ by
$$
\mathscr{P}_W := \{\nu \in \mathscr{P}_\X: \int_\X W(\bx) \nu (d\bx) < \infty\}.
$$
The weighted total variation metric $\rho_W$ on $\mathscr{P}_W$ is defined for any $\nu_1, \nu_2\in \mathscr{P}_W$ by
$$
\rho_W (\nu_1, \nu_2) =\sup_{|\varphi|_W \leq 1} \int_\X \varphi(x) (\nu_1 (d\bx)  - \nu_2(d\bx)),
$$
where the semi-norm $|\varphi|_W$ is defined by
$$
|\varphi|_W = \sup_{\bx\in \X} \frac{|\varphi(\bx)|}{1 + W(\bx)}.
$$
The convergence to the equilibrium result is stated in the theorem below.

\begin{thm}\label{thm:ergodicity}
The Boltzmann-Gibbs measure $\pi(d\bq\,d\bp)$ is the unique invariant measure of the dynamics \eqref{eq:LG}. Moreover, for every $a\in (0,\beta)$, there exists
$W\in C^2(\X,(0,\infty))$ and constants $\lambda, C> 0$ such that
$$
W(\bx) = \exp(a H(\bx) +\Psi(\bx)),
$$
with $\Psi(\bx) = o(H(\bx))$ as $H(\bx) \gt \infty$
and such that for every $\nu \in \mathscr{P}_W $ and $t \geq 0$
\be\label{eq:expdecay}
\rho_W(\nu P_t, \pi) \leq Ce^{-\lambda t}\rho_W(\nu, \pi).
\en
\end{thm}

\begin{rem}
The function $W$ in the above statement is in fact the Lyapunov function to be constructed below.  Theorem \ref{thm:ergodicity} implies an exponential convergence to the equilibrium of the Markov semigroup $P_t$ for almost optimal observables with sharp integrability. Indeed, first Theorem \ref{thm:ergodicity}
states that \eqref{eq:expdecay} holds with any initial condition $\nu$ satisfying
$$
\int_\X e^{aH(\bx)} \nu (d\bx) < \infty.
$$
Furthermore, consider $\nu = \delta_\bx$ and a test function $\varphi(x) :\X\gt\R$ with $|\varphi|_W < \infty$. Then \eqref{eq:expdecay} translates into
$$
|P_t \varphi(\bx) - \pi \varphi| \leq (1 + W(x))e^{-\lambda t} \rho_W(\nu, \pi) \leq C(1 + W(x))e^{-\lambda t}.
$$
One can not expect this hold when $a\geq \beta$ since in this case
there exist test functions  such that $\pi \varphi = \infty$ while
$\sup_\bx |\varphi(\bx)|/(e^{a H(\bx)}) < \infty$.
\end{rem}

The proof of Theorem \ref{thm:ergodicity} is based on
two ingredients:
  a local minorization condition on
  the probability density $P_t(x, \cdot)$ and the existence of a suitable Lyapunov function.
  The proof of the former is quite standard and follows from the support theorem \cite{stroock1972support, stroock1973probability}.
  This relies on the fact that diffusion of consideration here is hypoelliptic so that the probability density $P_t(x, \cdot)$
  is smooth by the H\"ormander’s theorem \cite{hormander1967hypoelliptic}.

The second ingredient, i.e.  the existence of a suitable Lyapunov is established in Proposition \ref{prop:Lyapunov}.
Its proof requires a more delicate analysis and is the major contribution of the present paper; see the next two sections.
The complete proof of the Theorem \ref{thm:ergodicity} can be found in Section \ref{sec:proof}.

\begin{prop}\label{prop:Lyapunov}
For any $a\in (0,\beta)$, there exists $W, \Psi\in C^2(\X, (0,\infty))$ satisfying the following:

(i) $W(\bx) = e^{a H(\bx ) + \Psi(\bx)}$ and $\Psi(\bx) = o(H(\bx))$ as $H(\bx) \gt\infty$. As a consequence, $W(\bx) \gt \infty$ as $H(\bx) \gt\infty$.

(ii) There exists $\lambda, C>0$ depending on $N$ such that
\be\label{ineq:L}
\mL W(\bx) \leq -\lambda W(\bx) + C.
\en
\end{prop}

\section{Construction of Lyapunov function: Building intuition from a simple example}\label{sec:model1}
In this section, we would like to build some intuition of constructing a Lyapunov function
for \eqref{eq:LG} through  examining carefully the structure of a simple two-body Coulomb system. To be more concrete, we consider a simple system of Coulomb gases with only two particles ($N=2$ in \eqref{eq:LG}). To simplify the dynamics even further,
we artificially freeze one of the particle at the origin by setting $q_1 = p_1 = 0$ and only study the dynamics of $(q,p) = (q_2,p_2)\in \X =  (\R^{d})^2\setminus (0,0)$ with $d\geq 2$:

\be\label{eq:ex}
\begin{aligned}
		\frac{dq}{dt} & = p\\
		\frac{dp}{dt} & = -\gamma p - \nabla U(q) + \sqrt{2\gamma \beta^{-1}}\frac{dB}{dt},
		\end{aligned}
\en
with $U(q) =  V(q) + K(q)$. Recall the Coulomb kernel $K$ in
\eqref{eq:Coulomb}. In this case, the generator $\mathcal{L}$ associated to above is
$$
\mL = p\cdot \nabla_q  - \nabla_q U(q) \cdot \nabla_p -\gamma p\cdot \nabla_p +
\gamma\beta^{-1} \Delta_p =: \mH + \gamma \mO,
$$
where $\mH$ is the conservative part (corresponding to the conservative Hamiltonian transport dynamics)
and $\mO$ is the dissipative part  (corresponding to the fluctuation/dissipation of the Orstein-Uhlenbeck diffusion process on the momentum):
$$
\mH = p\cdot \nabla_q  - \nabla_q U(q) \cdot \nabla_p  \text{ and }
\mO = -  p\cdot \nabla_p + \beta^{-1} \Delta_p.
$$

\subsection{Revisting old ideas}
An initial natural guess of the Lyapunov function $W$ would simply be the Hamiltonian
$H(q, p) = \frac{1}{2} |p|^2 + V(q) + K(q)$. However, a simply calculation yields that
\be\label{eq:LH}
\mL H = -\gamma |p|^2 + d\gamma \beta^{-1}.
\en
According to the assumptions on $V$ and $K$, the functions $p \gt |p|^2$ and $(q, p)\gt H(q,p)$ are not comparable
when either $|p|\gt \infty$  or $|p| \gt 0$ while $p$ remains bounded.
Consequently, the function $|p|^2$ can not be bounded from below by $H$ and hence the inequality \eqref{eq:Lyapunov} fails to hold.
This is mainly because the dissipative operator $\mO$ only acts in the $p$-direction.
To transfer the dissipation from $p$ to $q$ one needs to modify the Lyapunov function from $W = H$ to something else.

In the case where $U(q)$ has no singularity and grows at least quadratically at infinity,
a Lyapunov function of the form
$$
W(q,p) = H(q,p) + c p\cdot q
$$
 was constructed previously (see e.g \cite{talay2002stochastic, mattingly2002ergodicity}). By choosing $c>0$ small enough,
 one can justify that the function $H(q,p)$ and $W(q,p)$ defined above are comparable when $|p|^2 + |q|^2 \gt\infty$.
 The trick of adding a $c p\cdot q$ term is essential to the proof of hypercoercivity of kinetic equations
 \cite{dric2009hypocoercivity, dolbeault2015hypocoercivity} which shows that the generator is coercive with respect to some modified norms that are equivalent to $H^1$ or $L^2$ (weighted against the invariant measure $\pi$) although it is not coercive with respect to the corresponding standard norms.

 Why adding a correction term $c p\cdot q$ term with $H$ leads to a Lyapunov function
 was explained in detail in \cite{cooke2017geometric, hairer2009slow}. The key observation there is that at high energy level ($H \gg 1$), the
 system \eqref{eq:ex} evolves approximately around the deterministic orbit of the Hamiltonian dynamics ($\gamma  = 0$ in \eqref{eq:ex}).
 Therefore the average dissipative effect  is approximately given by
 $$
 - \gamma \langle |p|^2 \rangle (H(q, p)) + d\gamma \beta^{-1}
 $$
 where $\langle |p|^2 \rangle (h)$ is the average of $|p|^2$ along the deterministic Hamiltonian orbits at energy level $h$.
To correct $-|p|^2 + d\gamma \beta^{-1}$ on the right side of
\eqref{eq:LH} to the average dissipation as above, a Lyapunov function of the form $W = H + \Psi$ was constructed where
the corrector $\Psi$ solves the following Poisson equation:
\be\label{eq:poisson}
\mH \Psi (q, p) = -\gamma (|p|^2  - \langle |p|^2 \rangle (H(q, p)) ).
\en
When $U(q)$ has no singularity and grows at least quadratically at infinity, it was shown that $\Psi(q,p) \sim cp\cdot q$ as $|q| \gt \infty$.
The same averaging idea was further exploited by \cite{cooke2017geometric} to construct Lyapunov function for the Langevin dynamics with a Lennard-Jones type singularity
\be\label{eq:LJ}
U(q) = A|q|^{a} + \frac{B}{|q|^{b}}
\en
 for some constants $a, b > 0$. In this case, the  operator $\mH$ used in the Poisson equation \eqref{eq:poisson}
 to define the corrector has to be replaced by another effective operator to take account of the blow-up of $H$ as $q \gt 0$; see \cite[Section 4]{cooke2017geometric} for more details.
 However, the idea of reducing stochastic dynamics to the deterministic Hamiltonian dynamics
 seems to be difficult to be generalized to more complicated singular potentials $U(q)$,
 such as \eqref{eq:U}.

 Recently, the authors of \cite{herzog2017ergodicity}
 proposed a different idea to construct a Lyapunov function and used it to prove the
 geometric ergodicity of Langevin dynamics with a larger class of admissible singular potentials $U$. More specifically,
 their admissible condition on $U$ states that   for any sequence $\{q_k \}\subset \mD$ with $U(q_k) \gt\infty$,
 \be\label{eq:adm}
 |\nabla U(q_k)| \gt \infty  \text{ and } \frac{|\nabla^2 U(q_k)|}{|\nabla U(q_k)|^2} \gt 0.
 \en
Note that this condition is fullffiled by both the Lennard-Jones type
potential defined in \eqref{eq:LJ} and the Coulomb potential in $\R^d$
with $d\geq 3$. However, the second limit fails to hold for two
dimensional Coulomb potential and the Log potential in one dimension.  Like in our the present paper, the choice of Lyapunov function in
\cite{herzog2017ergodicity} was inspired by the dominating scaling at
the singularity of the kernel (see the discussion below) and the
scaling analysis which inspired the hypocoercivity constructions in \cite{dolbeault2015hypocoercivity}. However, the choice made in this paper is able to treat all Coulomb potentials in $\R^d$ for any dimension $d\geq 2$ uniformly. Moreover, the new Lyapunov function constructed here satisfies a Lyapunov condition with constants that have explicit dependence in $N$ (see Remark \ref{rem:Ndep}). The latter may be helpful for studying  scaling limit  problems as $N\gt\infty$. A more geometric understanding of why this new Lyapunov function is better still eludes us.

  Since our construction of Lyapunov function for the Coulomb system follows closely that of \cite{herzog2017ergodicity}, it is worth to briefly review the techniques therein.   Their construction relies on the following key observation (which was described earlier in \cite{dolbeault2015hypocoercivity}):
 under the following parabolic scaling
 $$
 (t, q, p) \gt (\lambda^{-2}t,\lambda^{-1}q, p )
 $$
 as $\lambda \gt 0$, the dynamics \eqref{eq:ex} formally reduces to its overdamped limit
 $$
 \frac{dq}{dt} = -\nabla U(q) + \sqrt{2\beta^{-1}}\frac{dB}{dt}.
 $$
  when the potential $U(q)$ is also  properly scaled up to infinity as $|q|\gt \infty$ or $|q|\gt 0$.
  This observation is nice because the overdamped Langevin dynamics above has a natural Lyapunov function, namely the potential $U(q)$.
 As a result, effectively it suffices to consider building a Lyapunov function by focusing on the regime where $U$ (instead of $H$) is large while $p$ remains bounded.
 Based on such observation, the authors of \cite{herzog2017ergodicity}
  proposed a Lyapunov function of the form
  \be\label{eq:W1}
  W(q,p) = e^{b(H(q,p) + \Psi(q,p))} ,
  \en
  where $b > 0$ and $\Psi(q,p) = o(H(q,p))$ as $H(q,p) \gt\infty$.
 With this definition, assuming $\Psi \in C^2(\X)$, one obtains that
  \be\label{eq:LW1}
  \mL W = bW (\mL (H(q,p) + \Psi(q,p)) + b\gamma\beta^{-1} |\nabla_p (H(q,p) + \Psi(q,p))|^2).
  \en
  In order to let $W$ satisfy the Lyapunov condition \eqref{eq:Lyapunov}, the corrector $\Psi$ has to be chosen
  appropriately in order to gain dissipation in $q$ when
  $U(q)$ is large. In particular, it is necessary to have for some $C>0$
  \be\label{eq:Lpsi}
  \mL \Psi \leq -C  \text{ when } U \geq R \gg 1.
  \en
    For this end, one need to consider scaling large $U$ in two different routes: $|q|\gt \infty$ and $|q| \gt 0$.
  For the Lennard-Jones potential defined in \eqref{eq:LJ}, it was shown
  that in both scaling regimes the dominating term in the generator is given by the following operator
  $$
  \A  = -\nabla U(q)\cdot \nabla_p.
  $$
  Hence effectively the condition \eqref{eq:Lpsi} becomes
  $$
   \A \Psi = -\nabla U(q)\cdot \nabla_p\Psi   \leq -C  \text{ when } U \geq R \gg 1.
  $$
  A typical choice of $\Psi$ suggested by \cite{herzog2017ergodicity} is a function $\Psi \in C^2(\X)$  and
  $$
  \Psi = \kappa \frac{p\cdot \nabla U(q)}{|\nabla U(q)|^2}  \text{ on } \X \cap \{U \geq R\}
  $$
  for some large constant $\kappa > 0$ which leads to $\A \Psi = -\kappa$ on $ \X \cap \{U \geq R\}$. Therefore on $ \X \cap \{U \geq R\}$, it holds that
  \be\label{eq:bdLH}
  \bea
&   \mL (H + \Psi) + b\gamma \beta^{-1}|\nabla_p (H(q,p) + \Psi(q,p))|^2 \\
& = \mL H + \A \Psi + (\mL - \A)\Psi +  b\gamma \beta^{-1}|\nabla_p (H(q,p) + \Psi(q,p))|^2  \\
  & = -\gamma |p|^2 + d \gamma \beta^{-1} - \kappa - p\cdot \nabla_q \Psi(q,p) -\gamma p\cdot \nabla_p \Psi(q,p) + b\gamma\beta^{-1} |p + \nabla_p \Psi(q,p)|^2\\
  & = -\gamma (1  - b\beta^{-1}) |p|^2 + d \gamma \beta^{-1} - \kappa +\kappa ( 2b\gamma\beta^{-1} - 1) p^T \nabla_q \Big(\frac{ \nabla U(q)}{|\nabla U(q)|^2} \Big) p\\
  & \qquad - \gamma \kappa \frac{p\cdot \nabla U(q)}{|\nabla U(q)|^2} + \frac{b\gamma\beta^{-1}\kappa^2}{|\nabla U(q)|^2} \\
  & \leq -\underbrace{\Big(\gamma(1  - b\beta^{-1} )- \kappa | 2b\gamma\beta^{-1} - 1| \Big|\nabla_q \Big(\frac{ \nabla U(q)}{|\nabla U(q)|^2} \Big) \Big|\Big)}_{=:\delta} |p|^2 + d \gamma \beta^{-1}- \kappa \\
  & \qquad + \frac{\delta}{2} |p|^2 +
     \frac{\gamma^2\kappa^2}{2\delta |\nabla U(q)|^2}+ \frac{b\gamma\beta^{-1}\kappa^2}{|\nabla U(q)|^2}\\
     & = -\frac{\delta}{2} |p|^2 + d \gamma \beta^{-1}- \kappa +
     \Big(\frac{\gamma^2\kappa^2}{2\delta} + b\gamma\beta^{-1} \kappa^2\Big) \frac{1}{|\nabla U(q)|^2},
  \ena
  \en
  where the inequality above follows from the Young's inequality.
  Now we fix $\kappa = 3d \gamma   \beta^{-1}$ and $b< \beta^{-1}$. Then by condition \eqref{eq:adm},
  there exists $R > 0$ large enough such that  $\delta > 0$ and that
  $$
  \Big(\frac{\gamma^2\kappa^2}{2\delta} + b\gamma\beta^{-1} \kappa^2\Big) \frac{1}{|\nabla U(q)|^2} \leq  d\gamma \beta^{-1} \text{ on } \X \cap \{U \geq R\},
  $$
  which implies that
   $$
   \bea \mL (H + \Psi) + b\gamma \beta^{-1}|\nabla_p (H(q,p) + \Psi(q,p))|^2 \leq  -d \gamma \beta^{-1} \text{ on } \X \cap \{U \geq R\}.
   \ena
   $$
   On the other hand,  on $ \X \cap \{U \leq R\}$, the right side of \eqref{eq:bdLH} is bounded. By combining  the above estimates and \eqref{eq:LW1}, one sees that the function  $W$  defined in \eqref{eq:W1} satisfies the Lyapunov condition \eqref{eq:Lyapunov}. Again this relies on  the crucial assumption \eqref{eq:adm} on $U(q)$. However, this assumption fails for two dimensional Coulomb potential or one dimensional Log potential.

\subsection{Lyapunov function for the two-body system} We now build a new Lyapunov function for the  two-body system \eqref{eq:ex} with a Coulomb potential $K$ in $\R^d$ for any $d\geq 2$.  Similar to the previous discussion, we
would like to consider scalings for $q$ at large $U$. Since the main issue comes from the singularity of
the kernel, we first restrict our attention on the regime $|q| \gt 0$. In particular, we consider the scaling
$$
(q, p) = (\lambda^{-1} Q, P)
$$
with $\lambda \gt \infty$. Then the generator $\mL$, when rewritten in the $(Q, P)$-coordinates,
takes the form
\be
\bea
\mL_{(Q, P)}& = -\lambda P \cdot \nabla_Q - (\nabla V(\lambda^{-1} Q) + \nabla K (\lambda^{-1} Q) ) \cdot \nabla_P - \gamma P \cdot \nabla_P + \gamma\beta^{-1} \Delta_P\\
 & \approx \begin{cases}
           \lambda P\cdot \nabla_Q  + \lambda \frac{Q}{|Q|^2}\cdot \nabla_P & \text{ if } d=2,\\
            \lambda^{d-1} \frac{Q}{|Q|^d}\cdot \nabla_P & \text{ if } d \geq 3,
           \end{cases}
\ena
\en
where in the last line we have used the assumption that $V(q)\in C^2(\R^d)$ and hence finite when $|q|$ is bounded and that the gradient of the Coulomb kernel is
$
\nabla K (q) = \frac{q}{|q|^d}
$
for any  $d\geq 2$. Therefore the dominating term of the generator in the regime $|q| \gt 0$ is
\be\label{eq:L2}
\A = \begin{cases}
      p\cdot \nabla_q + \frac{q}{|q|^2} \cdot \nabla_p & \text{ if } d=2,\\
      \frac{q}{|q|^d}\cdot \nabla_p & \text{ if } d \geq 3.
     \end{cases}
\en
Motivated by this, we define a Lyapunov function similar to \eqref{eq:W1}:
\be\label{eq:W2}
W(q, p) = \exp(a H(q,p)+ \Psi(q,p))
\en
with $a  > 0$. We would like to pick $\Psi$ so that $\mL \Psi \approx \A \Psi$ becomes very negative when $U(q)\gt\infty$ in the route $|q|\gt0$.
In view of the prefactors in front of the derivation operators in \eqref{eq:L2}, it is natural to consider $\Psi$ with the property that
\be\label{eq:psi1}
\Psi(q,p) \approx -b\,\frac{p\cdot q}{|q|^\alpha} \text{ as } |q| \gt 0
\en
for some constants $b,\alpha > 0$.
In fact, with this choice, we have from \eqref{eq:L2} that
$$
\mL \Psi(q,p) \approx  \A \Psi =\begin{cases}
                                 -\frac{b}{|q|^{\alpha+2}}(|p|^2 |q|^2 - \alpha |p\cdot q|^2)  -b \frac{1}{|q|^\alpha} & \text{ if } d=2, \\
                                 - \frac{b}{|q|^{d-2+\alpha}}& \text{ if } d \geq 3.
                                \end{cases}
$$
In addition, by setting $0 < \alpha \leq 1$, one sees that the first term on the right side of above in the case $d =2$ is always
non-positive and hence obtains that
$$
\mL \Psi(q, p) \lesssim -\frac{1}{|q|^{d-2+\alpha}} \text{ as } |q| \gt 0.
$$
In the subsequent discussion, we adopt the choice $\alpha=1$ for the purpose of simplifying
calculations.

To obtain similar behavior about $\mL \Psi$
in the other regime of large $U(q)$, namely $|q| \gt \infty$, we simply add a $cp\cdot q$ term in our $\Psi$
since by the discussion in the last section this term gives extra dissipation in $q$-direction when $|q|$ is large.
To sum up, the considerations above lead to the following new choice of corrector
\be\label{eq:psi2}
\Psi(q,p) =- b \,  \frac{p\cdot q}{|q|} + c\, p\cdot q.
\en
Notice that we have set $\alpha = 1$ in \eqref{eq:psi1}.
With this $\Psi$, we verify that the function $W$ defined \eqref{eq:L2} indeed
satisfies the Lyapunov condition \eqref{eq:Lyapunov} with appropriate parameters $a,b,c$.
Indeed, recall the Hamiltonian $H(q, p) = \frac{1}{2}|p|^2 + V(q) + K(q) $. By definition,
$$
|\Psi (q, p) | \leq b |p| + c|p||q| \text{ for every } (q,p)\in \X.
$$
In view of the growth condition Assumption \ref{ass:v} (ii) on $V(q)$, it is clear that when $c > 0$ is sufficiently small,
$
\Psi(q,p) = o(H(q,p))
$
as  $H \gt \infty$ and thus $W(q,p) = \exp(H(q,p)(1 + o(1)))$.
This in particular implies that $W \gt \infty$ as $H\gt \infty$.

Next we show that $W$ satisfies the inequality \eqref{eq:Lyapunov}.  A straightforward calculation yields
$$
\bea
\frac{\mL W(q,p)}{W(q,p)} & = p \cdot  \nabla_q \Psi(q,p) -\gamma p \cdot  (a \nabla_p H(q,p) + \nabla_p \Psi(q,p))\\
& - \nabla_q U(q) \cdot\nabla_p \Psi(q,p) + \gamma\beta^{-1} |a \nabla_p H(q,p) + \nabla_p \Psi(q,p)|^2 + da\gamma \beta^{-1}.
\ena
$$

To evaluate the terms appearing in the above, we need the following ansatze:

$$
\bea
\nabla_q H(q, p) & = \nabla_q U(q) = \nabla V(q) - \frac{q}{|q|^d},\qquad \nabla_p H(q, p)  = p,\\
\nabla_q \Psi(q, p)& =  -b\frac{p}{|q|} + b(p\cdot q)\frac{q}{|q|^3} + cp,\qquad \nabla_p \Psi(q, p) = -b\frac{q}{|q|} + cq.
\ena
$$
Consequently,
\ben
\begin{aligned}
& \frac{\mL W(q,p)}{W(q,p)}  = -\frac{b}{|q|^3}\big( |p|^2|q|^2 - (p\cdot q)^2\big) + c|p|^2  - a\gamma |p|^2 + b\gamma \frac{p\cdot q}{|q|} -c\gamma p\cdot q\\
 & \qquad - \big(-b\frac{\nabla V(q) \cdot q}{|q|} + c  \nabla V(q)\cdot q + \frac{b}{|q|^{d-1}} - \frac{c}{|q|^{d-2}}\big) \\
 & \qquad + \gamma\beta^{-1}\big(a^2|p|^2 - 2ab \frac{p\cdot q}{|q|} + 2ac p\cdot q+ b^2 - 2bc|q| + c^2|q|^2 \big)+ d\gamma\beta^{-1}\\
  & = \underbrace{-\frac{b}{|q|^3}\big( |p|^2|q|^2 - (p\cdot q)^2\big)}_{I_1} + \underbrace{(c - a\gamma (1 - a\beta^{-1})) |p|^2 + (b\gamma - 2ab\gamma\beta^{-1}) \frac{p\cdot q}{|q|}}_{I_2}  \\
  & \qquad + \underbrace{\Big(\frac{b}{|q|} - c\Big) \nabla V(q) \cdot q}_{I_3} + \underbrace{(2ac\gamma\beta^{-1} - c\gamma)p\cdot q  - 2bc |q| + c^2 |q|^2}_{I_4}\\
  &\qquad \underbrace{- \Big(\frac{b}{|q|^{d-1}} - \frac{c}{|q|^{d-2}}\Big) + b^2\gamma\beta^{-1} + da\gamma\beta^{-1}}_{I_5}.\\
\end{aligned}
\enn
Let us fix $a\in (0, \beta) $ and $ b > 0$. Then
$
I_1 \leq 0
$.
Moreover, if $c < - a\gamma (1 - a\beta^{-1})/2$, then
$$
I_2 \leq \frac{- a\gamma (1 - a\beta^{-1})}{2} |p|^2 + |b\gamma -2ab\gamma\beta^{-1}||p|.
$$
 In addition, thanks to Assumption \eqref{ass:v}, we have
 $$
\bea
I_3 &= \Big(\frac{b}{|q|} - c\Big) \nabla V(q) \cdot q \ (\mathbf{1}_{|q| > \frac{2b}{c}} + \mathbf{1}_{|q| \leq \frac{2b}{c}}) \\
& \leq -\frac{c\cdot c_3}{2} |q|^2 + \frac{c\cdot c_4}{2} + \max_{|q|\leq \frac{2b}{c}} (b |\nabla V(q)| + c|q||\nabla V(q)|).
 \ena
 $$
 By Young's inequality, it holds for any $\eps > 0$ that
$$
\bea
I_4
\leq \eps |p|^2 + \frac{c^2\gamma^2 |2a\beta^{-1} -1|^2 }{4\eps} |q|^2 +2c^2|q|^2 + b^2 .
\ena
$$
Lastly, one has
$$\bea
I_5 & = - \Big(\frac{b}{|q|^{d-1}}  - \frac{c}{|q|^{d-2}}\Big)(\mathbf{1}_{|q| < \frac{b}{2c}} + \mathbf{1}_{|q| \geq \frac{b}{2c}}) + b^2\gamma\beta^{-1} + d\gamma\beta^{-1}\\
& \leq -\frac{b}{2|q|^{d-1}} + c\,\Big(\frac{2c}{b}\Big)^{d-2} + b^2\gamma\beta^{-1} + da\gamma\beta^{-1}.
\ena
$$
From above, one sees that if one first sets $\eps < \frac{a\gamma (1 - a\beta^{-1})}{2} $ and then
chooses $c>0 $ sufficiently small so that the terms involving $|q|^2$  in $I_4$ can be absorbed by those in $I_3$,
then one obtains that
\be\label{eq:LW}
\frac{\mL W}{W} \leq -c_5(|p|^2 + |q|^2) -\frac{c_6}{|q|^{d-1}} + c_7
\en
for some $c_5, c_6, c_7 > 0$. It is clear that
$W(q,p) \gt \infty$ leads to one of the following options: $|p|\gt\infty, |q|\gt\infty$ or $|q|\gt0$. The latter
implies by \eqref{eq:LW} that there exists $\lambda > 0$ such that
$$
\mL W \leq -\lambda W \text{ on } \X\cap \{W(q,p) \geq R\}
$$
when $R$ is sufficiently large. On the other hand, it follows from \eqref{eq:LW} that $\mL W$ is bounded  when  $W (q, p) \leq R $.
 Hence we have shown that there exist constants $\lambda, K > 0$   such that
$$
\mL W \leq -\lambda W + K.
$$
\section{Construction of Lyapunov function: General case}
\label{sec:model2}
This section devotes to the proof of Proposition \ref{prop:Lyapunov}, i.e. building a Lyapunov function
for the dynamics of Coulomb gases with any finite number of particles ($N \geq 2$).
Informed by the analysis of the two-body case in the previous section, we define with constants $b,c>0$ the correction term
\be\label{eq:Psi}
\bea
\Psi(\bq, \bp) & = -\frac{b}{N} \Big(\sum_{i=1}^N p_i \cdot \big(\sum_{j\neq i}^N\frac{q_i - q_j}{|q_i - q_j|}\big)\Big) + c\,\bp\cdot \bq\\
& =   -\frac{b}{N}  \sum_{1\leq i\neq j \leq N} \frac{(p_i - p_j) \cdot (q_i - q_j)}{|q_i - q_j|} + c\, \bp \cdot \bq.
\ena
\en
Note that the function above reduces to \eqref{eq:psi2} when $(\bq,
\bp) = ( (q, 0), (p, 0))$ and $N=2$. Hence, we can understand this
choice as the many-particle analog of model problem presented in Section~\ref{sec:model1}.
We also define with $a > 0$ the Lyapunov function
\be\label{eq:W3}
W(\bq, \bp) = \exp(a H(\bq,\bp)+ \Psi(\bq,\bp)).
\en
By the same reasoning as in the previous section, one sees that when $c>0$ is sufficiently small,
$
\Psi(\bq,\bp) = o(H(\bq, \bp))
$
and hence $W(\bq,\bp) = \exp(aH(\bq,\bp)(1 + o(1)))$. This implies that $W \gt \infty$ as $H\gt \infty$.

Next we continue to show that $W$ satisfies \eqref{eq:Lyapunov}. In fact,
$$
\bea
\frac{\mL W(\bq,\bp)}{W(\bq,\bp)} & = \bp \cdot \nabla_\bq \Psi(\bq,\bp) - \gamma \bp \cdot  (a \nabla_\bp H(\bq,\bp) + \nabla_\bp \Psi(\bq,\bp))\\
& - \nabla_\bq U(\bq) \cdot\nabla_\bp \Psi(\bq,\bp) +  \gamma\beta^{-1} |a \nabla_\bp H(\bq,\bp) + \nabla_\bp \Psi(\bq,\bp)|^2 + Nda\gamma \beta^{-1}.
\ena
$$
We need the following derivatives:
$$
\bea
\partial_{q_i} H(\bq, \bp) & = \partial_{q_i} U(\bq) = \nabla V(q_i) - \frac{1}{N}
\sum_{j\neq i}^N \frac{q_i - q_j}{|q_i - q_j|^d},\\
\partial_{q_i} \Psi(\bq, \bp)& =   -\frac{b}{N} \sum_{j\neq i}^N \Big(\frac{p_i - p_j}{|q_i - q_j|} - \frac{(p_i - p_j)\cdot (q_i - q_j)}{|q_i - q_j|^3} \, (q_i - q_j)\Big) + cp_i,\\
\partial_{p_i} H(\bq, \bp) & = p_i,\\
\partial_{p_i} \Psi(\bq, \bp)& =  -\frac{b}{N} \sum_{j\neq i}^N \frac{q_i - q_j}{|q_i - q_j|} + cq_i.
\ena
$$
It follows  that
\ben
\begin{aligned}
 & \bp \cdot \nabla_\bq \Psi(\bq,\bp)   \\
 & =  -\frac{b}{N}  \sum_{i=1}^N p_i \cdot \Big(\sum_{j\neq i}^N \Big(\frac{p_i - p_j}{|q_i - q_j|} - \frac{(p_i - p_j)\cdot (q_i - q_j)}{|q_i - q_j|^3} \, (q_i - q_j)\Big) \Big) + c|\bp|^2\\
 & =   \underbrace{-\frac{b}{N} \sum_{1\leq i\neq j\leq N}^N \Big(\frac{|p_i - p_j|^2}{|q_i - q_j|} -  \frac{|(p_i - p_j)\cdot (q_i - q_j)|^2}{|q_i - q_j|^3}\Big)}_{\leq 0} +  c|\bp|^2 \\
 & \leq c|\bp|^2,\\
& -\gamma \bp \cdot   (a \nabla_\bp H(\bq,\bp) + \nabla_\bp \Psi(\bq,\bp)) = -\gamma a |\bp|^2 - \frac{\gamma b}{N} \sum_{1\leq i\neq j\leq N}^N \frac{(p_i - p_j)\cdot(q_i - q_j)}{|q_i - q_j|} \\
& \qquad - \gamma c\, \bp \cdot \bq.
\end{aligned}
\enn
Moreover,
\be\label{eq:gupsi}
\begin{aligned}
&- \nabla_\bq U(\bq) \cdot\nabla_\bp \Psi(\bq,\bp)\\
&  = - \sum_{i=1}^N\Big( \nabla V(q_i) - \frac{1}{N}
\sum_{j\neq i}^N \frac{q_i - q_j}{|q_i - q_j|^d}\Big) \cdot \Big(-\frac{b}{N} \sum_{j\neq i}^N \frac{q_i - q_j}{|q_i - q_j|} + cq_i\Big)\\
& =  \frac{b}{N} \sum_{1\leq j\neq i\leq N}^N \frac{(\nabla V(q_i) - \nabla V(q_j))\cdot (q_i - q_j)}{|q_i - q_j|}
- c\sum_{i=1}^N \nabla V(q_i)\cdot q_i \\
&\qquad - \frac{b}{N^2}\sum_{i=1}^N \Big( \sum_{j\neq i}^N\frac{q_i - q_j}{|q_i - q_j|^d}\Big)\cdot\Big( \sum_{k\neq i}^N \frac{q_i - q_k}{|q_i - q_k|}\Big)
  + \frac{c}{N} \sum_{1\leq i \neq j\leq N} \frac{1}{|q_i - q_j|^{d-2}}.\\
\end{aligned}
\en
By Assumption \ref{ass:v}, the first term on the second line of above can be bounded by
$$
\bea
 & 2b\sum_{i=1}^N |\nabla V(q_i)| - c  \sum_{i=1}^N \nabla V(q_i)\cdot q_i \\
 & \leq 2b\sum_{i=1}^N |\nabla V(q_i)| -cc_2\sum_{i=1}^N V(q_i) + cMN\\
 & \leq -\frac{c\,c_2}{2} \sum_{i=1}^N  V(q_i) + N\Big(cM + 2bM_{\frac{cc_2}{4b}} \Big)\\
 & \leq -\frac{c\,c_1c_2}{2} |\bq|^2+ N\Big(cM(1 + \frac{c_1}{2}) + 2bM_{\frac{cc_2}{4b}}\Big),
\ena$$
where we have used \eqref{eq:assv3} with $\eps = \frac{c c_2}{4b}$ in the penultimate line and \eqref{eq:assv1} in the last line above.
Thanks to Lemma \ref{lem:bdT}, the third line of \eqref{eq:gupsi} can be bounded by
$$
\bea
& -\frac{b}{N^2} \sum_{1\leq i\neq j\leq N} \frac{1}{|q_i - q_j|^{d-1}} + \frac{c}{N} \sum_{1\leq i \neq j\leq N} \frac{1}{|q_i - q_j|^{d-2}}\\
& \qquad - \frac{b}{N^2}\sum_{1\leq i\neq j\leq N} \frac{1}{|q_i - q_j|^{d-1}} \Big(1 - \frac{cN}{b} |q_i - q_j|\Big) \Big(\mathbf{1}_{|q_i - q_j| \leq \frac{b}{2Nc}} +  \mathbf{1}_{|q_i - q_j| \geq \frac{b}{2Nc}}\Big)\\
& \leq -\frac{b}{2N^2}\sum_{1\leq i\neq j\leq N} \frac{1}{|q_i - q_j|^{d-1}} +  cN \Big(\frac{2Nc}{b}\Big)^{d-2}.
\ena
$$
It then follow from the last two estimates that
$$
\bea
 & -\nabla_\bq U(\bq) \cdot\nabla_\bp \Psi(\bq,\bp) \\
 & \leq -\frac{c\,c_1c_2}{2} |\bq|^2+ N\Big(cM(1 + \frac{c_1}{2}) + 2bM_{\frac{cc_2}{4b}} + c\Big(\frac{2Nc}{b}\Big)^{d-2}\Big)- \frac{b}{2N^2}\sum_{1\leq i\neq j\leq N} \frac{1}{|q_i - q_j|^{d-1}}.
 \ena
$$
In addition, it holds that
\ben \begin{aligned}
& \gamma\beta^{-1} |a \nabla_\bp H(\bq,\bp) + \nabla_\bp \Psi(\bq,\bp)|^2  = \gamma \beta^{-1}\sum_{i=1}^N \Big|ap_i +  \frac{b}{N} \Big(\sum_{j\neq i}\frac{q_i - q_j}{|q_i - q_j|}\Big) + cq_i\Big|^2\\
 & = \gamma\beta^{-1}\Big( a^2 |\bp|^2 + 2 a\gamma \sum_{i=1}^N p_i\cdot \Big( \frac{b}{N} \Big(\sum_{j\neq i}\frac{q_i - q_j}{|q_i - q_j|}\Big) + cq_i \Big) + \sum_{i=1}^N\Big|\frac{b}{N} \Big(\sum_{j\neq i}\frac{q_i - q_j}{|q_i - q_j|}\Big) + cq_i\Big|^2\Big)\\
 & \leq \gamma\beta^{-1}\Big( a^2 |\bp|^2 + \frac{2ab}{N} \sum_{i\neq j}^N \frac{(p_i - p_j)\cdot(q_i - q_j)}{|q_i - q_j|}
  + 2ac \bp\cdot \bq +  \frac{2b^2}{N^2} \sum_{i=1}^N  \Big|\sum_{j\neq i}\frac{q_i - q_j}{|q_i - q_j|}\Big|^2 + 2c^2|\bq|^2\Big)\\
  & \leq \gamma\beta^{-1}\Big( a^2 |\bp|^2 + \frac{2ab}{N} \sum_{i\neq j}^N \frac{(p_i - p_j)\cdot(q_i - q_j)}{|q_i - q_j|}
  + 2ac\bp\cdot \bq + \frac{2N(N-1)^2b^2}{N^2} + 2c^2 |\bq|^2\Big).
\end{aligned}
\enn
Combining the  estimates above yields
$$
\bea
\frac{\mL W(\bq,\bp)}{W(\bq,\bp)} & \leq
\big(c - \gamma a (1 - a\beta^{-1})\big) |\bp|^2 + \gamma c(2a\beta^{-1} - 1 )\bp \cdot \bq
- c\,\big(\frac{c_1c_2}{2}-2c\big)|\bq|^2\\
&\qquad  + \gamma b|1 - 4a\beta^{-1}|N^{1/2}|\bp|
-\frac{b}{2N^2}\sum_{1\leq i\neq j\leq N} \frac{1}{|q_i - q_j|^{d-1}}\\
& \qquad + N\Big(cM(1 + \frac{c_1}{2}) + 2bM_{\frac{cc_2}{4b}} + c\Big(\frac{2Nc}{b}\Big)^{d-2} + 2b^2\gamma\beta^{-1}\Big)
\\
&\leq \big(c +\eps_1 + \eps_2 - \gamma a (1 - a\beta^{-1})\big) |\bp|^2
+ \Big(\frac{\gamma^2 c^2 |2a\beta^{-1}-1|^2}{4\eps_1} - c\,\big(\frac{c_1c_2}{2}-2c\big)\Big) |\bq|^2\\
& - \frac{b}{2N^2}\sum_{1\leq i\neq j\leq N} \frac{1}{|q_i - q_j|^{d-1}}\\
& + N\Big(cM(1 + \frac{c_1}{2}) + 2bM_{\frac{cc_2}{4b}} + c\Big(\frac{2Nc}{b}\Big)^{d-2} + 2b^2\gamma \beta^{-1} + \frac{\gamma^2b^2 |1-4a\beta^{-1}|^2}{4\eps_2}\Big).
\ena
$$
Now we choose the parameters $a, b, c,\eps_1,\eps_2 > 0$ such that
\be\label{eq:parameter}
\bea
& 0< a < \beta,\\
& b > 0,\\
& c +\eps_1 + \eps_2 - \gamma a (1 - a\beta^{-1}) < 0,\\
& \frac{\gamma^2 c^2 |2a\beta^{-1}-1|^2}{4\eps_1} - c\,\big(\frac{c_1c_2}{2}-2c\big) < 0.
\ena
\en
Note that the last two inequality hold in particular when $\eps_1 + \eps_2 \leq  \gamma a (1 - a\beta^{-1})/2$ and $c$ is sufficiently small.
With above choice of parameters, we have obtained that there exist $\alpha, C= C(N) > 0$ such that
\be\label{eq:LW2}
\frac{\mL W(\bq,\bp)}{W(\bq,\bp)} \leq  -\alpha (|\bq|^2 + |\bp|^2) - \frac{b}{2N^2}\sum_{1\leq i\neq j\leq N} \frac{1}{|q_i - q_j|^{d-1}}
+ C(N).
\en
Observe from the derivarion above that the constant $C(N) = \mathcal{O}(N^{d-1})$.
By the same reasoning as in the previous section, the above inequality implies that there exists $\lambda > 0$ such that when $R \gg 1$,
$$
\mL W \leq -\lambda W \text{ on } \X \cap \{W(\bq, \bp) \geq R\}.
$$
On the other hand, it follows from \eqref{eq:LW2} that $\mL W$ is bounded
on the set $ \X \cap \{W(\bq, \bp) \geq R\}$. To conclude, we have
\be\label{eq:lambdaC}
\mL W \leq -\lambda W  + C
\en
for some constant $C$ depending on $N$.

\begin{rem}\label{rem:Ndep}
  The dependence on $N$ of the constants appearing in the Lyapunov condition above can be made explicit. In fact, let us define   for $\kappa > 0$ the set
  $$
  \mathcal{C}_\kappa := \{(\bq, \bp)\in \X: |\bq|\leq \kappa N^{\frac{d-1}{2}}, |\bp|\leq \kappa N^{\frac{d-1}{2}} \text{ and } \min_{i\neq j} |q_i-q_j| \geq \kappa^{-1} N^{-\frac{d+1}{d-1}} \}.
  $$
  From the estimate \eqref{eq:LW2}, one sees that if $(\bq, \bp)\in \X\setminus \mathcal{C}_\kappa$ with some $\kappa \gg 1$, then
  $$
\mL W\leq -\lambda(N) W
  $$
  holds with $\lambda(N)  = C_\kappa N^{d-1}$ for some $C_\kappa > 0$. However, due to the exponential form of $W$, it follows from \eqref{eq:LW2} that
  $$
\mL W(\bq,\bp) \leq C(N) \max_{(\bq,\bp)\in \mathcal{C}_\kappa} W((\bq,\bp)) \lesssim \mathcal{O} (e^{N^{\alpha + d +2}})\ \forall (\bq, \bp)\in \mathcal{C}
  $$
   with $\alpha > 0$ depending on the growth of $V$. Combining above leads to
  \eqref{eq:lambdaC} with $\lambda = \lambda(N)  = \mathcal{O}(N^{d-1})$ and $C = C(N) =   \mathcal{O} (e^{N^{\alpha + d +2}})$.
\end{rem}

\begin{lem}\label{lem:bdT}
 For any $\bq \in \mathcal{D} = \{\bq = (q_1,\cdots, q_n)^T \in (\R^{d})^N: q_i \neq q_j\ \forall i,j = 1,2\cdots, n \text{ with } i \neq j\}$, the following inequality holds:
 \be\label{eq:Tq}
 J(\bq) := \sum_{i=1}^N \Big( \sum_{j\neq i}^N\frac{q_i - q_j}{|q_i - q_j|^d}\Big) \cdot \Big( \sum_{k\neq i}^N \frac{q_i - q_k}{|q_i - q_k|}\Big) \geq \sum_{1\leq i\neq j\leq N} \frac{1}{|q_i - q_j|^{d-1}} .
 \en
\end{lem}

\begin{proof}
It is straightforward to see that
$$
\bea
& J(\bq)  = \sum_{i=1}^N \sum_{j\neq i}^N \frac{1}{|q_i  - q_j|^{d-1}} +  \sum_{i=1}^N \sum_{\substack{1\leq j< k\leq N\\ j, k\neq i}} \frac{(q_i - q_j) \cdot (q_i -q_k)}{|q_i - q_j|^d |q_i -q_k|}\\
& \qquad +  \sum_{i=1}^N \sum_{\substack{1\leq k < j\leq N\\ j, k\neq i}} \frac{(q_i - q_j) \cdot (q_i -q_k)}{|q_i - q_j|^d |q_i -q_k|}\\
& = \sum_{i=1}^N \sum_{j\neq i}^N \frac{1}{|q_i  - q_j|^{d-1}}\\
& \qquad +  \sum_{i=1}^N \sum_{\substack{1\leq j< k\leq N\\ j, k\neq i}} \frac{(q_i - q_j) \cdot (q_i -q_k)}{|q_i - q_j| |q_i -q_k|}\Big(\frac{1}{|q_i - q_j|^{d-1}} +\frac{1}{|q_i - q_k|^{d-1}} \Big)\\
& =:  J_1 + J_2,
\ena
$$
where the second equality follows from swapping the indices $k$ and $j$ in the summation  in the second line of above.

We show that $J_2$ is non-negative. In fact, $J_2$ can  be expanded as
$$
\bea
J_2 & =  \sum_{i=1}^N \sum_{\substack{1\leq j< k\leq N\\ j, k\neq i}} \frac{(q_i - q_j) \cdot (q_i -q_k)}{|q_i - q_j| |q_i -q_k|}\Big(\frac{1}{|q_i - q_j|^{d-1}} +\frac{1}{|q_i - q_k|^{d-1}} \Big)\\
& =  \Big(\sum_{1\leq i < j < k \leq N}
+  \sum_{1\leq j < i < k \leq N}  +  \sum_{1\leq j < k < i \leq N}\Big)\\
& \qquad \qquad \Big\{\frac{(q_i - q_j) \cdot (q_i -q_k)}{|q_i - q_j| |q_i -q_k|}\Big(\frac{1}{|q_i - q_j|^{d-1}} +\frac{1}{|q_i - q_k|^{d-1}} \Big)\Big\}.
\ena
$$

By relabelling the indices $i \gt j, j\gt i, k\gt k$ in the second term on the right side and
$i \gt k, j\gt i, k\gt j$ in the third term, one gets
$$
\bea
J_2 & = \sum_{1\leq i < j < k \leq N} \Big\{  \frac{(q_i - q_j) \cdot (q_i -q_k)}{|q_i - q_j| |q_i -q_k|}\,\Big(\frac{1}{|q_i - q_j|^{d-1}} +\frac{1}{|q_i - q_k|^{d-1}} \Big)\\
& \qquad + \frac{(q_j-q_i)\cdot (q_j -q_k)}{|q_j - q_i| | q_j - q_k|} \,\Big(\frac{1}{|q_j-q_i|^{d-1}} + \frac{1}{|q_j - q_k|^{d-1}}\Big)\\
& \qquad  + \frac{(q_k - q_i)\cdot(q_k - q_j)}{|q_k - q_i| |q_k - q_j|}\Big(\frac{1}{|q_k - q_i|^{d-1}} +  \frac{1}{|q_k - q_j|^{d-1}}\Big)
\Big\}.
\ena
$$
We claim that the individual terms inside the summation above are non-negative. Indeed, let us define for fixed indices $i<j<k$ the angles $\{\theta_m\}_{m=1}^3$ as follows:
$$
\bea
\cos \theta_1 &=  \frac{(q_i - q_j) \cdot (q_i -q_k)}{|q_i - q_j| |q_i -q_k|}, \cos \theta_2 = \frac{(q_j-q_i)\cdot (q_j -q_k)}{|q_j - q_i| | q_j - q_k|},\\
\cos \theta_3 &= \frac{(q_k - q_i)\cdot(q_k - q_j)}{|q_k - q_i| |q_k - q_j|}.
\ena
$$
It suffices to show that $S \geq 0$ with $S$ defined by
$$
\bea
S & = \cos\theta_1 \Big(\frac{1}{|q_i - q_j|^{d-1}} +\frac{1}{|q_i - q_k|^{d-1}} \Big) + \cos \theta_2\Big(\frac{1}{|q_j-q_i|^{d-1}} + \frac{1}{|q_j - q_k|^{d-1}}\Big) \\
& \qquad + \cos \theta_3 \Big(\frac{1}{|q_k - q_i|^{d-1}} +  \frac{1}{|q_k - q_j|^{d-1}}\Big).
\ena
$$
We prove this by considering the following two cases.

{ Case 1: The three points $\{q_i, q_j,q_k\}$ lie on a common line.  }
Without loss of generality we may assume that $q_i$ lies in the middle of the segment connecting $q_j$ and $q_k$.
In this case, we have $\cos \theta_1=-1, \cos\theta_2 = \cos\theta_3 = 1$. This leads to
$$
S \geq \frac{2}{|q_j  - q_k|^{d-1}} \geq 0.
$$

{ Case 2: The three points $\{q_i, q_j,q_k\}$ do not lie on a common line. } In this case, the there points form a triangle with edges of lengths $|q_i  - q_j |, |q_i - q_k|$ and $|q_j  - q_k|$ and with angles $\{\theta_m\}_{m=1}^3$. Hence it holds that $\theta_1 + \theta_2 + \theta_3 = \pi$ and $0 < \theta_i < \pi,i=1,2,3$.
If the triangle is an acute or right triangle, then it is obvious that $S\geq 0$ since $\cos (\theta_i) \geq 0$ for every $i=1,2,3$.
On the other hand,  if the triangle is an obtuse triangle,
then one of the angles, say $\theta_1$, is larger than $\pi/2$.
As a result, we have $\theta_2 + \theta_3 < \pi/2$ and thus
$\cos \theta_2 \geq \cos(\theta_2 + \theta_3) = -\cos(\theta_1) > 0$
and $\cos \theta_2 \geq \cos(\theta_2 + \theta_3) -\cos(\theta_1) > 0$. Consequently,
$$
S \geq - \cos \theta_1 \Big(\frac{2}{|q_j-q_k|^{d-1}}\Big) > 0.
$$
Summarizing the  discussion above yields $J_2 \geq 0$.
It follows that
$$
T(q) \geq J_1 =  \sum_{1\leq i\neq j\leq N} \frac{1}{|q_i - q_j|^{d-1}}.
$$
\end{proof}
A similar result to Lemma \ref{lem:bdT} was proved in \cite[Lemma 5.2]{bolley2018dynamics}.

\section{Proof of Main Results}\label{sec:proof}
\subsection{Proof of Proposition \ref{prop:exist}}
Proposition \ref{prop:exist} follows immediately from the following stronger result.

\begin{prop}\label{prop:pw}
For any initial point $\bx  = (\bq, \bp) \in \X$,
there exits a unique strong pathwise solution $\bx(t)$ to the equation \eqref{eq:LG}. Furthermore, for $a\in (0,\beta)$, let $W$ be the Lyapunov function given in Proposition \ref{prop:Lyapunov} with constants $\lambda, C$ defined in \eqref{ineq:L}. Then for every $\bx = (\bq, \bp) \in \X$ and $t\geq 0$,
\be\label{eq:PW}
P_t W(\bx) \leq e^{-\lambda t} W(\bx) +  \frac{C}{\lambda}.
\en
\end{prop}

\begin{proof}
For $R > 0 $, let us define stopping times
$$
\tau_{R} = \inf \{ t\geq 0: W(\bx_t) > R\}, \tau = \lim_{R\gt\infty} \tau_{R} \text{ and } \tau_{\partial \X} = \inf \{ t\geq 0: \bx_t\in \partial \X\}.
$$
Thanks to the coercivity of the function $W(\bx)$ from Lemma \ref{lem:coercive}, we have the inclusion $\{\tau = \infty\}\subset \{\tau_{\partial \X} = \infty\}$.
As a result, equation \eqref{eq:LG} has a strong pathwise solution $\bx_t$ up to the stopping time $\tau_R$.
We want to show that $\mathbf{P}_\bx (\tau  = \infty) = 1$.

In fact, an application of Ito's formula to the function $f(\bx, t) = e^{\lambda t} (W(\bx) - \frac{C}{\lambda})$ yields
$$
\mathbf{E}_\bx [e^{t\wedge \tau_R}\cdot W(\bx_{t\wedge \tau_R})] \leq W(\bx) + \frac{Ce^{\lambda t}}{\lambda}.
$$
Now consider a fixed $t > 0$ and letting $R\gt\infty$. One can obtain from above and the fact that $W\geq 0$ that
\be\label{eq:RP}
R \cdot \mathbf{P}_\bx (\tau_R \leq t) \leq W(\bx) + \frac{Ce^{\lambda t}}{\lambda},
\en
which implies that
$$
\mathbf{P}_\bx (\tau \leq t) = \lim_{R\gt\infty} \mathbf{P}_\bx (\tau_R \leq t)
 \leq \lim_{R\gt\infty} \frac{1}{R} \Big(W(\bx) + \frac{Ce^{\lambda t}}{\lambda}\Big) = 0.
$$
Thus we have obtained that $\mathbf{P}_\bx (\tau_R > t) = 1$ for any finite $t> 0$.
This in turn implies that the particles system $\bx_t$ do not explode or collide in finite time.

Finally, using again the estimate \eqref{eq:RP}, one has
$$
e^{\lambda t} \mathbf{E}_\bx[W(\bx(t)) \mathbf{1}_{\tau_R \geq t}] \leq W(\bx) + \frac{Ce^{\lambda t}}{\lambda},
$$
which leads to \eqref{eq:PW} by letting $R\gt\infty$ since $\lim_{R\gt\infty}\mathbf{P}_\bx ({\tau_R \geq t}) = 1$.

\end{proof}

\begin{rem}
A similar stopping time argument using the Hamiltonian $H$  as a Lyapunov function instead of $W$ would lead to
$$
\mathbf{E}_\bx H(\bx_t) \leq H(\bx) + Ct
$$
and hence also prevents explosion as well as collision leading to global well-posedness of solution. However, an estimate of the form \eqref{eq:PW} would not be available.
\end{rem}

\begin{lem}[Coercivity of $W$ and $H$]\label{lem:coercive}
Let $H(\bq, \bp)$ and $W(\bq, \bp)$ be given by \eqref{eq:H}
and \eqref{eq:W3} with parameters chosen as in \eqref{eq:parameter}.  Then
$$
\lim_{(\bq,\bp) \gt \partial \X}  W(\bq, \bp) = \lim_{(\bq,\bp) \gt \partial \X}  H(\bq, \bp) = +\infty.
$$
\end{lem}
\begin{proof}
Since we have seen in the previous section that $W(\bq, \bp) = e^{a H(\bq,\bp) (1+ o(1))}$, we only need to show the above is valid for $H$.
 For this, it suffices to  show that for any $R > 0$, there exists $R_1 > 0,R_2 > 0, r > 0$ such that
 $H(\bq, \bp) \geq R$ provided that one of
 the following holds: $|\bp| \geq R_1$, or  $|\bq| \geq R_2$ or $\min_{i\neq j} |q_i -q_j| \leq r$.

 This is obvious for $d \geq 3$. In fact, since both $K$ and $V$ are non-negative and $V$ grows at least
 quadratically at infinity by Assumption \ref{ass:v}, in view of the expression
$$
\bea
H(\bq, \bp) =  \frac{1}{2} |\bp|^2 + \sum_{i=1}^N V(q_i) + \frac{1}{2N} \sum_{1\leq i \neq j \leq N} K(q_i -q_j),
\ena
$$
the claim above is valid if one picks for example $R_1 = \sqrt{2R}, R_2 = \sqrt{\frac{M+R}{c_1}}, r = (\frac{1}{2NR})^{\frac{1}{d-2}}$.

For $d= 2$, we follow \cite[Lemma 3.2]{bolley2018dynamics}. In fact, it follows from the simple inequality $\log x \leq x$ for any  $x > 0$ that
$$\bea
H(\bq, \bp) & \geq \frac{1}{2} |\bp|^2 + c_1 |\bq|^2 -M - \frac{1}{2N} \sum_{1\leq i\neq j \leq N} |q_i - q_j|\\
& \geq  \frac{1}{2} |\bp|^2 + c_1 |\bq|^2 -M - \sum_{i=1}^N |q_i| \\
& \geq  \frac{1}{2} |\bp|^2 + \frac{c_1}{2} |\bq|^2 -M -\frac{N}{2c_1 },
\ena
$$
where we used the Young's inequality in the last line. Consequently, $H \geq R$ if either
$|\bp| \geq R_1 = \sqrt{2 (R+ M + \frac{N}{2c_1 })}$
or $|\bq| \geq R_2 = \sqrt{\frac{2}{c_1} (R+ M + \frac{N}{2c_1 })}$. Finally,
we only need to show that $H\geq R$ when $|\bp|\leq R_1, |\bq| \leq R_2$ and $|q_i -q_j|\leq r$ for some small $r > 0$ and for some $1\leq i\neq j\leq N$.
To see this, observe that for any $i\neq j$,
$$
H(\bq, \bp) \geq  - \frac{1}{N} \log |q_i -q_j| -
\frac{1}{2N} \sum_{\substack{1\leq k\neq \ell \leq N\\ \{k,\ell\} \neq \{i,j\}}} \log |q_k - q_\ell|
\geq  - \frac{1}{N} \log |q_i -q_j| - \sqrt{N} R_2,
$$
where the last line follows from
$$
\frac{1}{2N} \sum_{\substack{1\leq k\neq \ell \leq N\\ \{k,\ell\} \neq \{i,j\}}} \log |q_k - q_\ell| \leq  \sum_{k=1}^N|q_k| \leq \sqrt{N} |\bq| = \sqrt{N} R_2.
$$
Hence if $|q_i - q_j| \leq r = e^{-N(R + \sqrt{N}R_2)}$, then $H \geq R$. This completes the proof of the lemma.
\end{proof}

\subsection{Proof of Theorem \ref{thm:ergodicity}}
The proof of Theorem \ref{thm:ergodicity} is based on Proposition \ref{prop:Lyapunov} and a local minorization condition on the transition density $P_t(\bx, \cdot)$ as shown in the next proposition.

\begin{prop}\label{prop:minor}
For any $a\in (0,\beta)$, let $W$ be the Lyapunov function defined in Proposition \ref{prop:Lyapunov}. Define for $R > 0$ the compact set $\mathcal{C}_R = \{\bx\in \X: W(\bx) \leq R\}$. Then for any sufficiently large $R>0$ and  any $t_0 > 0$, there exists a probability measure $\nu$ on $\X$ and a constant $\alpha > 0$ such that for every Borel set $A\subset \X$ and every $\bx\in \mathcal{C}_R $,
\be\label{eq:minor}
P_{t_0} (\bx, A) \geq \alpha \nu (A).
\en
\end{prop}
\begin{proof}
The proof of the proposition is the same as the proof of \cite[Corollary 5.12]{herzog2017ergodicity} (see also \cite[Lemma 2.3]{mattingly2002ergodicity}) and hence omitted.
\end{proof}

\begin{prop}\label{prop:Feller}
For every $\bx\in \X$ and $t>0$, we have $\text{supp } P_t(\bx, \cdot) = \X$, where $\text{supp } \nu $ denotes
the support of the measure $\nu$. In addition, the measure  $ P_t(\bx, \cdot)$ is absolutely continuous with
respect to the Lebesgue measure on $\X$. Let $p_t(\bx,\by)$ be the associated  probability density. Then the mapping $(t, \bx, \by) \mapsto p_t(\bx,\by): (0,\infty) \times \X\times \X  \gt [0,\infty)$ is continuous.
\end{prop}
\begin{proof}
The proof of the proposition is the same as  the proof of \cite[Proposition]{herzog2017ergodicity} and hence omitted.
\end{proof}

\begin{proof}[Proof of Theorem \ref{thm:ergodicity}]
First it is straightforward to check that the measure $\pi$ is an invariant measure. The uniqueness follows from \cite[Theorem 3.16]{hairer2006ergodicity}  and the fact that $P_t$ is strong Feller and that $\text{supp } \nu = \X$ for any invariant measure $\nu$ as shown in Proposition \ref{prop:Feller}.

To show the exponential convergence to the equilibrium we apply \cite[Theorem 1.2]{hairer2011yet} to the sampled Markov chain on $\X$ with transition density $P_n(\bx, \cdot) := P_{n \Delta t}(\bx, \cdot)$ with $\Delta t > 0$ and $n\in \N$. In fact, Proposition \ref{prop:pw} and Proposition \ref{prop:minor} imply Assumption 1 and Assumption 2 in \cite{hairer2011yet}. As a result of \cite[Theorem 1.2]{hairer2011yet}, one obtains that there exists $C>0$ and $\delta\in (0,1)$ such that
$$
\rho_W( \nu_1 P_n,  \nu_2 P_n) \leq C\delta^n \rho_W(\nu_1,  \nu_2)
$$
for all $\nu_1, \nu_2\in \P_W$. To see the above bound holds for any $t$, let us write $t = n \Delta t + r$ with $r \in (0, \Delta t)$. For any measurable function  $\varphi: \X \gt \R$ with $|\varphi|_W \leq 1$, it follows from Proposition \ref{prop:pw} that
$$
|P_r \varphi| \leq |\varphi|_W\, \sup_\bx \frac{1 + P_r W(\bx)}{1 + W(\bx)} \leq C^\prime
$$
for some $C^\prime > 0$ independent of $r > 0$. As a consequence of Fubini theorem and Chapman-Kolmogorov equation, it holds that
$$
\rho_W (\nu_1 P_t, \nu_2 P_t) = \rho_W(\nu_1  P_n P_r, \nu_2 P_n P_r ) \leq  C^\prime \rho_W (\nu_1 P_n, \nu_2 P_n) \leq CC^\prime\delta^n \rho_W(\nu_1,  \nu_2).
$$
By defining $\eta = -\frac{\log \delta}{\Delta t}$ and $\tilde{C} = CC^\prime/\delta^{1/\Delta t}$, one obtains the desired bound
$$
\rho_W (\nu_1 P_t, \nu_2 P_t) \leq \tilde{C} e^{-\eta t} \rho_W(\nu_1, \nu_2).
$$
The inequality \eqref{eq:expdecay} follows by setting $\nu_2 = \pi$ and $\pi P_t = \pi$.

\end{proof}

\section{Extension and outlook}\label{sec:ext}
Although the present paper focuses on the Langevin particle system with Coulomb kernel, the main theorem \ref{thm:ergodicity} is still valid for
   Riesz-type repulsive kernels
   $$
   K(q) = |q|^{-s} \text{ with } 0 < s< d.
   $$
   The restriction $s < d$ is natural because otherwise the invariant measure is not normalizable and hence not a probability measure.
   In the Riesz case, Proposition \ref{prop:Lyapunov} still holds with the same Lyapunov function \eqref{eq:W3} and
   the same corrector $\Psi$ given by \eqref{eq:Psi}.
   The proof proceeds in the same manner as in Section \ref{sec:model2} as soon as one observes
   that the key inequality \eqref{eq:Tq} in  Lemma \ref{lem:bdT} is still valid with $d$ replaced by $s+2$.

For the same reason, our results also extend trivially to Langevin particles system \eqref{eq:LGN} in 1D in which the
kernel $K$ has Riesz or Log singularity. The only differece of the one dimensional system from systems in high dimensions is that
the domain $\mD$ of $U$ in this case is not pathwise connected. Therefore, the dynamics of the particles need to be restricted
to one of the connected components (depending on the initial configuration). For example, one may set the state space $\X = \mD\times \R^N$ with
$$
\mD = \{\bq = (q_1, q_2,\cdots, q_N) \in \R^N: q_1 < q_2 <\cdots < q_N\}.
$$
The convergence to the equilibrium result holds with the equilibrium measure $\pi$ restricted on $\X$ as well.

The convergence of Langevin dynamics with Coulomb kernel is studied via a probabilistic approach in this paper
based on a novel construction of Lyapunov function. The Lyapunov function satisfies
the Lyapunov condition \eqref{ineq:L} with the constants $\lambda, C$ depending on $N$ explicitly. However, the $N$-dependence of the constant $\alpha$ in the  minorization  condition \eqref{eq:minor} is very implicit. It remains an interesting question to obtain
an explicit estimate for the minorization condition so that one can obtain an $N$-explicit
convergence rate for the Langevin dynamics.

\section*{Acknowledgement}

Both authors acknowledge the support of the National Science Foundation through the grant DMS-1613337. YL  also acknowledges the partial support from the Duke Trinity College of Arts \& Sciences Office of the Dean.

\bibliographystyle{abbrv}
\bibliography{Coulomb}

\end{document}